\documentclass[12pt, a4paper]{amsart}
\usepackage[margin = 2.9cm]{geometry}
\usepackage[title]{appendix}
\usepackage{mathtools}
\usepackage{lineno}

\usepackage{subcaption}
\usepackage{amsmath}
\usepackage{systeme}
\usepackage{paralist}
\usepackage{graphics} 
\usepackage{epsfig} 
\usepackage{graphicx}
\usepackage{epstopdf}
 \usepackage[colorlinks=true]{hyperref}
\hypersetup{urlcolor=blue, citecolor=red}

\newtheorem{thm}{Theorem}[section]

\newtheorem{pp}{Proposition}
\newtheorem{claim}{Claim}

\theoremstyle{definition}

\newtheorem{remark}{Remark}

\newtheorem{example}{Example}
\newcommand*\diff{\mathop{}\!\mathrm{d}}

\DeclareMathOperator{\WF}{WF}

\begin{document}
\title{The X-ray transform on a generic family of smooth curves}
\author[Y. Zhang]{Yang Zhang}
\begin{abstract}
We study the X-ray transform over a generic family of smooth curves in $\mathbb{R}^2$ with a Riemannian metric $g$.
We show that the singularities cannot be recovered from local data in the presence of conjugate points, and therefore artifacts may arise in the reconstruction. 
We perform numerical
experiments to illustrate the results.
\end{abstract}

\maketitle

\section{Introduction}
In this work, we consider the integral transform over a generic family of smooth curves with conjugate points.
Let $M \subset \mathbb{R}^2$ be a bounded domain with a Riemannian metric $g$.
We consider a family $\Gamma$ of smooth curves satisfying the following properties.
\begin{itemize} 
    \item[(A1)] For each $(x,v) \in TM\setminus0$, there is exactly one unique curve $\gamma(t) \in \Gamma$ passing $x$ in the direction of $v$. We can assume $\gamma(0) = x$ and $\dot{\gamma}(0) = \mu(x,v) v$ with a smooth function $\mu(x,v) >0$ and denote it by $\gamma_{x,v}(t)$.
    \item[(A2)] Suppose $\gamma_{x,v}(t)$ depends on $(x,v)$ smoothly, and thus we assume 
    it solves a second order ODE 
    \begin{equation*}
    D_t\dot{\gamma}_{x,v}(t) = G(\gamma_{x,v},\dot{\gamma}_{x,v}),
    \end{equation*}
    where $D_t$ is the covariant derivative along $\dot{\gamma}$ and the generator $G$ is a smooth map invariantly defined on $TM$
    with the transformation law
    \[
    \widetilde{G}^i(\tilde{x}^i, \tilde{y}^i) = \frac{\partial \tilde{x}^i}{\partial x^j} G^j(x^i,y^i), \quad \text{ where }\tilde{y}^i = \frac{\partial \tilde{x}^i}{\partial x^j} y^j,
    \]
    under the coordinate change $x^i \mapsto \tilde{x}^i$ of $M$. 
\end{itemize}
We emphasize that the second property describes a large family of curves, including geodesics in Riemannian surfaces, geodesics in Finsler spaces, and magnetic geodesics, for more details see \cite[Chapter 4]{Bucataru2007}, \cite[Chapter 3]{Shen2013}, \cite{MR2351370}, and \cite{douglas1927general}.
For such a curve, we can freely shift the parameter but rescaling it may give us a different one, since the trajectory depends on the energy level.
This implies the first property is necessary for us to have a geodesic-like family of curves.

%
%

We define the integral transform along curves in $\Gamma$ as
\[
I_w f (\gamma)= \int w(\gamma(t),\dot{\gamma}(t)) f(\gamma(t)) dt, \quad \gamma \in \Gamma,
\]
for any distribution $f$ compactly supported in $M$.
This integral transform has been studied in \cite{Frigyik2007}.
It is shown that if $\Gamma$ is analytic and regular, i.e., if $G$ is analytic and $N^*\Gamma$ covers $T^*M$ without conjugate points,
then $I_w$ with an analytic and nonvanishing weight function is injective. A stability result for the localized normal operator is established there.
In \cite{Assylbekov2017}, the integral transform of functions and $1$-forms over such curves in an oriented Finsler surface is studied and the injectivity is proved if there are no conjugate points.

However, when conjugate points exist, in many cases the microlocal stability are lost.
The phenomenon of cancellation of singularities due to the existence of conjugate points are fully studied in the case of synthetic aperture radar imaging in \cite{MR3080199}, the X-ray transforms over a family of geodesic-like curves in \cite{MR2970707}, the geodesic ray transforms in \cite{Holman2015,Monard2015, Holman2017}. In particular, \cite{Monard2015} shows
that regardless of the type of the conjugate points, the geodesic ray transforms on Riemannian surfaces are
always unstable and we have loss of all derivatives, which leads to the artifacts in the reconstruction near
pairs of conjugate points.
The same phenomenon occur in the inversion of broken ray transforms with conjugate points, see \cite{Zhang2020}.
For more references about the integral transform over geodesics, magnetic geodesics, or a general family of curves, see \cite{Ainsworth, homan2017injectivity, kunstmann2021, Geometric,zhou2018local}.

This work is inspired by \cite{Monard2015} and we prove an analog of cancellation of singularities for $I_w$ in Theorem \ref{thm_C_diff} and \ref{thm_microkernel}.
We prove that the local problem is ill-posed if there are conjugate points, i.e., singularities
conormal to smooth curves in $\Gamma$ cannot be recovered uniquely. 
The microlocal kernel is described in Theorem \ref{thm_microkernel}.
In Section \ref{lambda}, we introduce $\lambda$-geodesics to describe this generic family of smooth curves and we define the conjugate points.
In Section \ref{sec_microlocal}, we consider the recovery of $f$ from the integral transform $I_w$ using the local data.
Here by the local data, we mean the integral transform is only known in a small neighborhood of a fixed curve $\gamma(t)  \in \Gamma$. 
The transform $I_w$ is an Fourier integral operator (FIO) with its canonical relation described in Proposition \ref{pp_C}.
We microlocalize the problem and prove Theorem \ref{thm_C_diff} and Theorem \ref{thm_microkernel}.
In Section \ref{sec_example}, we present several examples of such family of curves and show the existence of conjugate points.
In Section \ref{sec_num}, we illustrate the artifacts arising
in the reconstruction by numerical experiments. 
In particular, we show for the family of smooth curves in Example \ref{ex_c}, the singularities of $f$ can be recovered without artifacts if we consider the global data and have the prior knowledge that $f$ is a distribution with compact support. 

%
%


\subsection*{Acknowledgments}
The author would like to thank Plamen Stefanov for suggesting this problem and for numerous helpful discussions with him throughout this project, and to thank Gabriel Paternain and Gunther Uhlmann for helpful suggestions.
Part of this research was supported by NSF Grant DMS-1600327, the Simons Travel Grants, and was performed while the author was visiting IPAM.

\section{The $\lambda$-geodesics}\label{lambda}
\subsection{The generator $G$}
Locally we can rewrite the second order ODE as
\[
\ddot{\gamma}^i_{x,v}(t) = G^i(\gamma_{x,v}, \dot{\gamma}_{x,v}) - \dot{\gamma}_{x,v}^k\dot{\gamma}_{x,v}^l\Gamma_{kl}^i,  \quad i = 1, 2,
\]
where $\Gamma_{kl}^i$ is the Christoffel symbol.
The generator $G$ induces a vector field $\mathbf{G}$ on $TM$ given by 
\[
\mathbf{G} = y^i \frac{\partial }{\partial x^i} + (G^i(x, y) - y^ky^l \Gamma_{kl}^i)\frac{\partial }{\partial y^i} 
\]
in a local chart $\{x^i, y^i\}$ of $TM$. 
It preserves the form under coordinate changes by the following arguments and thus is invariantly defined, see also (30) in \cite{Frigyik2007}.
Indeed, consider a local coordinate chart $\{x^i, y^i\}$ near a fixed point $(x, y) \in TM$. 
Under the coordinate change $x^i \mapsto \tilde{x}^i$ of $M$, we have $\tilde{y}^i = ({\partial \tilde{x}^i}/{\partial x^j}) y^j$.
The double tangent bundle $T_{(x,y)} TM$ has a natural basis $\{{\partial }/{\partial {x}^i},{\partial }/{\partial {y}^i}\}$, which transforms as
\begin{align*}
&\frac{\partial }{\partial {x}^i} = \frac{\partial \tilde{x}^j}{\partial x^i} \frac{\partial }{\partial \tilde{x}^j} + \frac{\partial \tilde{y}^j}{\partial x^i} \frac{\partial }{\partial \tilde{y}^j}
= \frac{\partial \tilde{x}^j}{\partial x^i} \frac{\partial }{\partial \tilde{x}^j} + y^k \frac{\partial^2 \tilde{x}^j}{\partial x^i\partial x^k} \frac{\partial }{\partial \tilde{y}^j}, \\
&\frac{\partial }{\partial {y}^i} = \frac{\partial \tilde{x}^j}{\partial x^i} \frac{\partial }{\partial \tilde{y}^j}.
\end{align*}
In addition, the Christoffel symbol transforms as 
\[
{\Gamma}_{kl}^i = \frac{\partial x^i}{\partial \tilde{x}^m}\frac{\partial \tilde{x}^n}{\partial x^k} \frac{\partial \tilde{x}^p}{\partial x^l} \tilde{\Gamma}_{np}^m + \frac{\partial^2 \tilde{x}^m}{\partial x^k \partial x^l} \frac{\partial x^i}{\partial \tilde{x}^m}.
\]
We plug these transformation laws above in $\textbf{G}$ to have 
\begin{align*}
\textbf{G} &= y^i \frac{\partial }{\partial x^i} +  (G^i(x, y) - y^ky^l \Gamma_{kl}^i)\frac{\partial }{\partial y^i} \\
&= y^i\frac{\partial \tilde{x}^j}{\partial x^i} \frac{\partial }{\partial \tilde{x}^j} 
+ \big(y^iy^k \frac{\partial^2 \tilde{x}^j}{\partial x^i\partial x^k} + (G^i -  y^ky^l (\frac{\partial x^i}{\partial \tilde{x}^m}\frac{\partial \tilde{x}^n}{\partial x^k} \frac{\partial \tilde{x}^p}{\partial x^l} \tilde{\Gamma}_{np}^m + \frac{\partial^2 \tilde{x}^m}{\partial x^k \partial x^l} \frac{\partial x^i}{\partial \tilde{x}^m})) \frac{\partial \tilde{x}^j}{\partial x^i}\big) \frac{\partial }{\partial \tilde{y}^j}\\
&= \tilde{y}^j \frac{\partial }{\partial \tilde{x}^j}  + (\tilde{G}^j - \tilde{y}^k\tilde{y}^l \tilde{\Gamma}_{kl}^j)
\frac{\partial }{\partial \tilde{y}^j}.
\end{align*}
\subsection{Reparameterize $\gamma$}
Since $\dot{\gamma}$ never vanishes, we reparameterize these curves such that they have unit speeds w.r.t. the Riemannian metric $g$.
Note that this reparameterization will also change the generator $G(\gamma, \dot{\gamma})$ and the weight function in the integral transform.
In the following, we abuse the notation and use $\gamma_{x,v}(t)$ to denote the smooth curve with arc length passing $(x,v) \in SM$, where $SM$ is the unit circle bundle.
The family of smooth curves with this arc length parameterization solves a new ODE
\[
D_t \dot{\gamma} = \lambda(\gamma, \dot{\gamma}) \dot{\gamma}^{\perp},
\]
and are referred to as the $\lambda$-geodesics in \cite[Chapter 7]{Merry2011},
where $ \lambda(\gamma, \dot{\gamma}) \in C^{\infty}(SM)$ and $\dot{\gamma}^{\perp}$ is the vector normal to ${\gamma}^{\perp}$ with the same length and the rotation of $\pi/2$. 
When $\lambda$ is a smooth function on $M$, the corresponding $\lambda$-geodesic flow is called the magnetic flow, see \cite{ginzburg1996closed,MR2351370}.
When $\lambda$ is a smooth function on $SM$, it is called the thermostats, see \cite{dairbekov2007entropy}.

\subsection{Extend $\Gamma$}\label{subsection_extend}
Moreover, it is convenient to extend $\Gamma$,
such that for any $V = rv \in T_pM\setminus 0$, with  $r>0$ and $v\in S_pM$, there exists a unique curve $\gamma_{x,V}(t)$ belonging to $\Gamma$.
This can be done by defining $\gamma_{x,V}(t) = \gamma_{x,v}(rt)$.
We note that $\gamma_{x,V}(t)$ satisfies the following ODE
$$
D_t \dot{\gamma}_{x,V}(t) = |V|_g \lambda({\gamma}_{x,V}(t),\frac{1}{|V|_g}\dot{\gamma}_{x,V}(t)) \dot{\gamma}^\perp_{x,V}(t),
$$
which is equivalent to extend $\lambda(x, v)$ to $C^{\infty}(TM\setminus0)$.

As we mentioned before, for a curve satisfying the equation in Property (A2) in general, rescaling it might not give us the same one, i.e.,
$\gamma_{p,v}$ and $\gamma_{p,rv}$ are different curves.
However,  with the reparameterization and the extension above, the rescaling will cause no problem.

\subsection{Conjugate points.}
We define the \textit{exponential map} as  $\exp_p(t,v) =  \gamma_{p,v}(t)$ for $(p,v) \in SM$, see \cite{MR2970707}.
This definition uses polar coordinates and is independent of a change of the parameterization for curves in $\Gamma$.
For $v \in S_pM$, $t \in \mathbb{R}$, the map $V = t v \mapsto \exp_p(t,v)$ may not be smooth near $V = 0 \in T_pM$, see \cite[A.3]{MR2351370} for more details and properties.

By $d_{v} \exp_p(t,v)$, we always mean the differential of $\exp_p(t,v)$ w.r.t $v \in S_p M$.
This notation is different from the one $d_{v} \exp_p(v)$ for $v \in T_p M$ used in the case of geodesics.
We say a point $q = \gamma_{p,v}(t_0)$ is \textit{conjugate} to $p = \gamma_{p,v}(0)$ if the differential $d_{t,v} \exp_p(t,v)$ of the exponential map in polar coordinates has the rank less than the maximal one at $(t_0, v)$. 

To further describe the conjugate points, we consider the flow $\phi_t: SM \rightarrow SM$ given by
$
\phi_t(x,v) = (\gamma_{x,v}(t),\dot{\gamma}_{x,v}(t)).
$
For a fixed curve $\gamma_{p_0,v_0}(t)$, we set
\[
(p_1,v_1) = \phi_t(p_0,v_0), \quad (p_2,v_2) = \phi_{t+s}(p_0,v_0),
\]
where $(p_j,v_j) \in SM$, for $j = 0,1,2$.
Note that $\phi_t: SM \rightarrow SM$ is a local diffeomorphism, see \cite[Chapter 7]{Merry2011}.
Its pushforward $\diff \phi_t: T_{(p,v)} SM \rightarrow T_{\phi_t(p,v)} SM$ is invertible.
Let $\pi$ be the natural projection from $SM$ to $M$, which induces a map $\diff \pi$ from $T_{(p,v)} SM$ to $T_{p} M$.
We have the following proposition that relates the conjugate points with the pushforward of the flow $\phi_t$.
\begin{pp}\label{conjugate_J}
    A point $p_2 =\gamma_{p_0,v_0}(t+s) $ is conjugate to $p_1 =\gamma_{p_0,v_0}(t)$ if and only if there exists
    $\zeta_0 \in T_{(p_0, v_0)} SM$
    and $c \in \mathbb{R}$ satisfying
    \begin{equation}\label{zeta_0}
    \diff \pi (\diff \phi_t)(\zeta_0) = 0,\qquad
    \diff \pi (\diff \phi_{t+s})(\zeta_0) + c \dot{\gamma}_{p_0,v_0}(t+s) = 0.
    \end{equation}
\end{pp}
\begin{proof}
   Let $\zeta_j \in T_{(p_j,v_j)} SM$ for $j = 0, 1, 2$ such that
    \[
    \zeta_2 = \diff \phi_s (\zeta_1) = \diff \phi_{t+s} (\zeta_0),
    \]
    which implies $\zeta_1 = \diff \phi_t (\zeta_0)$ by the group property of $\phi_t$.
    Let $\{p_j^i, v_j^i\}$ be the local charts near $(p_j, v_j)$
    and they induce a natural local basis $\{\partial/\partial p_j^i, \partial/\partial v_j^i\}$ in $T_{(p_j, v_j)} SM$, for $j = 0,1,2$.
    We write $\zeta_j = (\alpha_j, \beta_j)$ in this local coordinates, where $\alpha_j$ can be identified as an element in $T_{p_j}M$ and
    $\beta_j$ in $T_{v_j} S_{p_j} M$.
    Note that $\diff \pi (\zeta_j) = \alpha_j$ for $j = 0,1,2$.
    Consider the representation of $\diff \phi_t$ in these local coordinates near $(p_j,v_j)$ given by the Jacobian matrix
    \[
    J_{(p_j,v_j,t)} =
    \begin{bmatrix}
    \begin{array}{cc}
    \diff_p\exp_{p}(t,v) & \diff_v\exp_{p}(t,v)\\
    \diff_p {\dot{\exp}_{p}(t,v)}&  \diff_v {\dot{\exp}_{p}(t,v)}
    \end{array}
    \end{bmatrix}_{(p_j,v_j,t)}.
    \]
    More explicitly, we have
    \begin{equation}\label{jacobieq1}
    \begin{bmatrix}
    \begin{array}{c}
    \alpha_2\\
    \beta_2
    \end{array}
    \end{bmatrix}
    = J_{(p_1,v_1,s)}
    \begin{bmatrix}
    \begin{array}{c}
    \alpha_1\\
    \beta_1
    \end{array}
    \end{bmatrix}
    =
    J_{(p_0,v_0,t+s)}
    \begin{bmatrix}
    \begin{array}{c}
    \alpha_0\\
    \beta_0
    \end{array}
    \end{bmatrix},
    \end{equation}
    and
    \begin{equation}\label{jacobieq1_2}
    \begin{bmatrix}
    \begin{array}{c}
    \alpha_1\\
    \beta_1
    \end{array}
    \end{bmatrix}
    = J_{(p_0,v_0,t)}
    \begin{bmatrix}
    \begin{array}{c}
    \alpha_0\\
    \beta_0
    \end{array}
    \end{bmatrix}.
    \end{equation}
    Recall that $p_2$ is conjugate to $p_1$ if only if there exists
    $c \in \mathbb{R}$ and $\beta_1 \in T_{v_1} S_{p_1} M$ such that
    \begin{align*}
    d_{s,v} \exp_{p_1}(s,v_1)(c, \beta_1) = 0 \Rightarrow d_v \exp_{p_1}(s,v_1)(\beta_1) = -c \dot{\gamma}_{p_1,v_1}(s),
    \end{align*}
    On the one hand, if such $c, \beta_1$ exist, then we pick $\zeta_1 = (0, \beta_1)$ and
    let $\zeta_0 = J_{(p_0,v_0,t)}^{-1} \zeta_1$.
    In this case,
    by equation (\ref{jacobieq1}) and (\ref{jacobieq1_2}) we have
    \begin{align*}
    & 0 = \alpha_1 = \diff \pi(\diff \phi_t)(\zeta_0) = d_p\exp_{p_0}(t,v_0)(\alpha_0) + d_v\exp_{p_0}(t,v_0)(\beta_0),\\
    & \alpha_2 = \diff \pi(\diff \phi_{t+s})(\zeta_0) =  d_p\exp_{p_1}(s,v_1)(\alpha_1) + d_v\exp_{p_1}(s,v_1)(\beta_1)= -c \dot{\gamma}_{p_1,v_1}(s),
    \end{align*}
    which proves (\ref{zeta_0}).
    On the other hand, if there is $\zeta_0$ such that (\ref{zeta_0}) is true, then we take $(\alpha_1, \beta_1) = J_{(p_0,v_0,t)}(\zeta_0)$.
    By (\ref{zeta_0}) we must have $\alpha_1 =0$ and then $d_v \exp_{p_1}(s,v_1)(\beta_1) = -c \dot{\gamma}_{p_1,v_1}(s)$, which implies that $p_2$ is conjugate to $p_1$.
%
\end{proof}
    In the following, recall some results in \cite[Chapter 4, 7]{Merry2011}.
We consider a moving frame $\{F(x,v), H(x,v), V(x,v)\}$ on $SM$, where $F(x,v)$ is the infinitesimal generator of $\phi_t$,
$H(x,v)$ is the horizontal vector field, and $V(x,v)$ is the vertical vector field.
For a detailed definition of $H$ and  $V$, see \cite[Section 4.1]{Merry2011}.
One can show that
\begin{align}\label{projection_frame}
    \diff \pi(F(x,v)) = v, \quad \diff \pi(H(x,v)) = v^\perp, \quad \diff \pi(V(x,v)) =0.
\end{align}
The $\lambda$-geodesic vector field $F$ is related to the geodesic vector field $X$ by
\[
F = X + \lambda V.
\]
One can show the commutators between them are given by
\begin{align*}
    &[H, V] = X = F- \lambda V,\\
    &[V, F] = H + V(\lambda) V,\\
    &[F, H] = (K - H(\lambda) + \lambda^2) V - \lambda F,
\end{align*}
where $K$ is the Gaussian curvature.

\begin{claim}\label{cl_1}
    For any constant $c$, we can find $\zeta_0' \in T_{(p_0, v_0)} SM$ such that
    \[
    \diff \pi(\diff \phi_{t})(\zeta_0') = c \dot{\gamma}_{p_0,v_0}(t).
    \]
\end{claim}
\begin{proof}
    By \cite[Section 7.1]{Merry2011}, we consider the infinitesimal generator $F(x,v)$ of the flow $\phi_t$ and  define
    $$Y(t) = c (\diff \phi_{-t})(F(\phi_t(x,v))).$$
    Differentiating both sides with respect to $t$ implies
    $$
    \dot{Y}(t) = c[F(x,v),F(x,v)] = 0.
    $$
    Thus $Y(t)$ is a constant vector field on $SM$. By denoting it as $Y = \zeta_0'$, we have  $ d\phi_{t}(\zeta_0') = cF(\phi_t(x,v))$.
    With $\diff \pi (F(x,v))  = v$, applying $\diff \pi$ to $ \diff \phi_{t}(\zeta_0')$ claims what we need.
\end{proof}
\subsection{Jacobi Fields}Inspired by Proposition \ref{conjugate_J}, we say a vector field $J_{(p_0,v_0)}(t)$ is a \textit{Jacobi field} along $\gamma_{p_0,v_0}$ if it can be written as
\begin{align}\label{jacobifield_rep}
J_{(p_0,v_0)}(t) = \diff \pi(\diff \phi_t)(\zeta) + (c_1 + c_2t) \dot{\gamma}_{p_0,v_0}(t),
\end{align}
for some $\zeta \in T_{(p_0, v_0)} SM$ and  $c_1, c_2 \in \mathbb{R}$.
We emphasize that the representation of a Jacobi field in form of (\ref{jacobifield_rep}) is not unique, which can be seen from Claim \ref{cl_1} and Proposition \ref{jocabifield_simplied}.
Moreover, this definition is different from the usual one that the Jacobi fields are defined as the variation field along $\gamma_{p_0,v_0}$,
for example see \cite{Assylbekov2017}.
However, they are essentially the same and the definition we choose here is more straightforward for the purpose of this work.
To understand the last term $t \dot{\gamma}_{p_0,v_0}(t)$, recall we extend $\Gamma$ in Section \ref{subsection_extend},
such that for any $V = rv \in T_pM\setminus 0$ with  $r>0$ and $v\in S_pM$, there exists a unique curve $\gamma_{x,V}(t)$ belonging to $\Gamma$.
\begin{pp}
    A point $p_2 = \gamma_{p_0,v_0}(t_2)$ is conjugate to $p_1 = \gamma_{p_0,v_0}(t_1)$ if and only if there exists a nonvanishing Jacobi field $J(t)$ along $\gamma_{p_0,v_0}$  such that $J(t_1) = J(t_2) = 0$.
\end{pp}
\begin{proof}
    If $p_2 = \gamma_{p_0,v_0}(t_2)$ is conjugate to $p_1 = \gamma_{p_0,v_0}(t_1)$,
    then by Proposition \ref{conjugate_J}, there exists $\zeta_0$ and $c$  such that
    \begin{equation*}
        \diff \pi (\diff \phi_{t_1})(\zeta_0) = 0,\qquad
        \diff \pi (\phi_{t_2})(\zeta_0) + c \dot{\gamma}_{p_0,v_0}(t_2) = 0.
    \end{equation*}
    Let $J_{(p_0,v_0)}(t) = \diff \pi(\diff \phi_t) (\zeta_0) + c \frac{t-t_1}{t_2 - t_1}  \dot{\gamma}_{p_0,v_0}(t)$.
    Then we have $J(t_1) = J(t_2) = 0$.

    Conversely, suppose there is a nonzero Jacobi field with $J_{(p_0,v_0)}(t_1)=J_{(p_0,v_0)}(t_2)=0$. More explicitly, we have
    \begin{align}\label{conjugate equation}
    \begin{split}
    &\diff \pi(\diff \phi_{t_1})(\zeta_0) + \lambda_1\dot{\gamma}_{p_0,v_0}(t_1) = 0, \\
    &\diff \pi(\diff \phi_{t_2})(\zeta_0) + \lambda_2\dot{\gamma}_{p_0,v_0}(t_2) = 0,
    \end{split}
    \end{align}
    where $\lambda_1 = c_1+c_2t_1,\lambda_2 = c_1+c_2t_2$.
    It suffices to find $\zeta_0'$ such that $\diff \pi(\diff \phi_{t_1})(\zeta_0') = \lambda_1 \dot{\gamma}_{p_0,v_0}(t)$.
     Such $\zeta_0'$ can be found by Claim \ref{cl_1}.
    In this way, the term with $\dot{\gamma}_{p_0,v_0}(t_1)$ in the first equation can replaced by considering $\tilde{\zeta}_0 = \zeta_0 + \zeta_0'$.
    By Proposition \ref{conjugate_J} we have $p_2$ is conjugate to $p_1$.
\end{proof}
The same conclusion is proved in \cite[Theorem 5.3]{Assylbekov2017} based on the definition of the Jacobi fields as the variational fields.

\subsection{The Jacobi equation.}
In this subsection, we show the following proposition in the case of $\lambda$-geodesic flow, analogous to \cite[Proposition 4.14 ]{Merry2011} in the case of the geodesic flow.
\begin{pp}[{\cite[Lemma 4.6]{Assylbekov2017}}]\label{dphi_t}
For any $\zeta \in T_{(p,v)} SM$, if we write $\diff \phi_t(\xi)$
\[
\diff \phi_t(\zeta) = x(t) F(\phi_t(x,v)) + y(t) H(\phi_t(x,v)) + z(t)V(\phi_t(x, v)),
\]
in the moving frame $\{F, H, V\}$,
then the smooth functions $x,y,z$ should satisfy
\begin{align*}
&\dot{x} = \lambda y,\\
&\dot{y} = z,\\
& \dot{z} = -({K-H(\lambda)  + \lambda^2}) y + V(\lambda) z.
\end{align*}
\end{pp}
Combining the last two equations above, one can see that $y$ should satisfy a second-order ODE.
Notice $\dot{\gamma}(t)^{\perp}, \dot{\gamma}(t)$ give us a moving frame along $\gamma_{p_0,v_0}(t)$.
This proves the following proposition analogous to \cite[Lemma 5.2]{Assylbekov2017}.
\begin{pp} \label{jacobi_eq}
If $J(t)$ is a Jacobi field along $\gamma_{p_0,v_0}$ expressed as
\[
J(t) = x(t) \dot{\gamma}_{p_0,v_0}(t) + y(t)\dot{\gamma}_{p_0,v_0}^\perp(t),
\]
then $x(t), y(t)$ satisfy the Jacobi equations
\begin{align}
&\dot{x} = \lambda y + c, \label{Jeq_x}\\
&\ddot{y} - V(\lambda) \dot{y} + (K-H(\lambda) + \lambda^2)y = 0, \label{Jeq_y}
\end{align}
where $c$ is a constant.
\end{pp}
\begin{remark}
If $\lambda=0$, we have the case of geodesics.
Compared to \cite[Lemma 5.2]{Assylbekov2017}, we have an extra constant $c$  in  (\ref{Jeq_x}), since by extending $\Gamma$ we include $t \dot{\gamma}(t)$ in our definition of the Jacobi field.
\end{remark}
\begin{proof}
Recall  $J(t)$ has the representation
\[
J_{(p_0,v_0)}(t) = \diff \pi(\diff \phi_t)(\zeta) + (c_1 + c_2t) \dot{\gamma}_{p_0,v_0}(t).
\]
Applying $\diff \pi$ to Proposition \ref{dphi_t}, we have the desired result by (\ref{projection_frame}).
\end{proof}

\section{Microlocal analysis of the local problem}\label{sec_microlocal}
\subsection{Parameterization of $\mathcal{M}$}
Since the group action of $\mathbb{R}$ on $M$ by $\phi_t$ is free and proper,
the family of curves form a smooth manifold $\mathcal{M} = SM/\phi_t$ of dimension $2\times 2-2=2$.
We parameterize $\mathcal{M}$ below near a fixed curve $\gamma \in \Gamma $, in the same way for the case of the geodesics in \cite[Section 3.1]{Monard2015}.
This parameterization gives us a local charts of the manifold of curves.

For this purpose,  for fixed $\gamma$ we choose a hypersurface $H$ such that $\gamma$ hits $H$ transversally at $p_0 = \gamma(0)$ and $v_0 = \dot{\gamma}(0)$.
Let $\mathcal{H}$ consist of all $(p, v) \in SM$ in a small neighborhood of $(p_0, v_0)$,
such that $p \in H$ and $v$ is transversal to $H$.
Then we can parameterize all curves near fixed $\gamma$
by the intersection point $p$ and the direction $v$, i.e., by an element $(p,v) \in \mathcal{H}$.
Suppose $H$ is locally given by $x^1=0$ and has coordinate $y$.
We can write each element in $\mathcal{H}$ using $(y, \eta)$ instead, where $\eta$ is the parameterization of $v \in S_{(0,y)}M$.
In the following, we use the new notation $\gamma(y, \eta, t)$ to denote the smooth curve starting from $p = (0,y)$ in the direction $v$.
In some cases, we write $\gamma(t), \frac{\partial \gamma}{\partial y}(t), \frac{\partial \gamma}{\partial \eta}(t)$ and omit the variables $y, \eta$ for simplification, if there is no confusion caused.
\begin{pp}\label{jocabifield_simplied}
    With this parameterization, the Jacobi field $J(t)$ along $\gamma(y, \eta, t)$ can be simplified as
    \[
    J(t) =  \kappa_1 \frac{\partial \gamma}{\partial y}(t)
    + \kappa_2 \frac{\partial \gamma}{\partial \eta}z(t) + c_1\dot{\gamma}(t) + c_2t \dot{\gamma}(t),
    \]
    where $\kappa_1, \kappa_2, c_1, c_2 \in \mathbb{R}$.
\end{pp}
\begin{proof}
    It is shown in Claim \ref{cl_1} that for any $c \in \mathbb{R}$, there exists $\zeta \in T_{(p,v)}SM $ such that
    $$
    S(t) = \diff \pi(\diff \phi_t)(\zeta) + c \dot{\gamma}(y, \eta, t) \equiv 0,
    $$
    where in local coordinates $p =(0,y)$ and we assume $c \neq 0$.
    On the other hand, in the local coordinates $(x^1,y, \eta)$ of $SM$, 
    the pushforward $\diff \phi_t$ has a representation and
    we can write
    \[
    \diff \pi(\diff \phi_t)(\zeta) = \kappa_1 \frac{\partial \gamma}{\partial y} +
    \kappa_2 \frac{\partial \gamma}{\partial x^1} + \kappa_3 \frac{\partial \gamma}{\partial \eta}.
    \]
    Then it suffices to show that $\kappa_2 \neq 0$ and therefore $\frac{\partial \gamma}{\partial x}$ can be written as a linear combination of $\frac{\partial \gamma}{\partial y},\frac{\partial \gamma}{\partial \eta},\dot{\gamma}$.
    Indeed, we have the following initial conditions
    \begin{equation}\label{eq_initial}
    \frac{\partial \gamma}{\partial \eta}(0) = 0, \quad (\frac{\partial \gamma}{\partial y}(0),\dot{\gamma}^\perp(0))_g \neq 0,
    \end{equation}
    where the first one comes from $\gamma(y, \eta, 0) \equiv p$ and the second one is from the assumption that $\gamma(y, \eta, t)$ hit $H$ transversally at $t=0$.
    Now we compute
    \[
    0 \equiv  (S(0),\dot{\gamma}^\perp(0))
    =\kappa_1 (\frac{\partial \gamma}{\partial y}(0),\dot{\gamma}^\perp(0))_g +
    \kappa_2 (\frac{\partial \gamma}{\partial x^1}(0),\dot{\gamma}^\perp(0))_g.
    \]
    It follows that $\kappa_2 \neq 0$, otherwise one has $\kappa_2 = \kappa_1 = 0$,
    which leads to \[
    \kappa_3 \frac{\partial \gamma}{\partial \eta} + c \dot{\gamma}(y, \eta, t) = 0.
    \]
    This contradicts with the initial condition in (\ref{eq_initial}) and that $\dot{\gamma}(y,\eta,t)$ is nonvanishing.
\end{proof}
\subsection{$I_w$ as an FIO}
The same analysis for the geodesic transform in \cite[Section 3.1]{Monard2015} works for $I_w$.
We briefly state these arguments for $I_w$ in this subsection.
For more details see \cite[Section 3.1]{Monard2015}.

We consider the curve-point relation
\[
Z =  \{(\gamma(y,\eta, t), x) \in \mathcal{M} \times M: \ x \in \gamma \},
\]
from the double fibration point of view of \cite{Gelextquotesinglefand1969,Helgason2013}.
Note that $Z$ is a smooth manifold
with the conormal bundle
\[
N^* Z = \{(\gamma, \hat{\gamma}, x, \xi)\in T^*(\mathcal{M} \times M)\setminus 0:  \   (\hat{\gamma}, \xi) \text{ vanishes on } T_{(\gamma, x)} Z_0\}.
\]
The integral transform along curves in $\Gamma$ defined by
\[
I_w f (\gamma)= \int w(\gamma(t),\dot{\gamma}(t)) f(\gamma(t)) \diff t, \quad f \in \mathcal{E}'(M),\  \gamma \in \Gamma.
\]
has a delta-type Schwartz kernel $w \delta_{Z}$, which is a conormal distribution supported in $N^*Z$.
Therefore, it is an FIO of order $-{1}/{2}$ and associated with the canonical relation $C = N^*Z'$, the twisted version of $N^*Z$ given by
\[
C = \{(\gamma, \hat{\gamma}, x, -\xi): \ \{(\gamma, \hat{\gamma}, x, \xi)\in N^*Z\}.
\]
By the local parameterization of $\mathcal{M}$, we can show $(y,\eta,\hat{y},\hat{\eta},x,\xi) \in C$ if and only if there exists $t$ with $x = \gamma(y,\eta,t)$ and
\begin{align}\label{conditioncanonical}
\xi_i \dot{\gamma}^i = 0,
\quad \xi_i \frac{\partial \gamma^i}{\partial y}= \hat{y},
\quad \xi_i \frac{\partial \gamma^i}{\partial \eta}= \hat{\eta},
\end{align}
where the Einstein summation convention is used.
\subsection{The description of $C$}
Inspired by \cite[Section 3.2]{Monard2015}, in this subsection, we write $C$ in a more explicit form and describe its properties.
Let \[
e_1(t) = \dot{\gamma}(t)^{\perp} , \quad e_2(t) = \dot{\gamma}(t)
\] be a moving frame.
We have the corresponding dual basis $\{e^1(t), e^2(t)\}$.
If we regard $\xi$ as a function of $t$, then the first condition in (\ref{conditioncanonical}) means
\[
\xi(t) = h(t) e^1(t),
\]
for some nonvanishing function $h(t)$.
Suppose we have the following expansion w.r.t the frame
\begin{align}
\frac{\partial \gamma}{\partial y} = a_1(t)e_1(t) + a_2(t) e_2(t), \quad
\frac{\partial \gamma}{\partial \eta} = b_1(t)e_1(t) + b_2(t) e_2(t).
\end{align}
The second and third conditions in (\ref{conditioncanonical}) imply
\begin{align*}\label{secondthirdcon}
h(t) a_1(t) = \hat{y}, \quad h(t) b_1(t) = \hat{\eta}.
\end{align*}
Therefore, the canonical relation has the local representation
\begin{align*}
C = \{ (y,\eta,
\underbrace{\mu a_1(t,y,\eta)}_{\hat{y}},
\underbrace{\mu b_1(t,y,\eta)}_{\hat{\eta}},
\gamma(t,y,\eta),
\mu \dot{\gamma}^\perp (t,y,\eta) ): \
 (y, \eta) \in \mathcal{H}, \ \mu \neq 0 \},
\end{align*}
where
\begin{equation}
a_1(t,y,\eta) = (e_1(t), \frac{\partial \gamma(t,y,\eta)}{\partial y})_g, \quad
b_1(t,y,\eta) = ( e_1(t), \frac{\partial \gamma(t,y,\eta)}{\partial \eta})_g \label{jeq_ab}
\end{equation}
are projections along $\dot{\gamma}^\perp(t,y,\eta)$.
The canonical relation $C$ can be parameterized by $(y,\eta, t, \mu )$ and therefore we have $\dim C = 4$.

\begin{pp}\label{pp_C}
    Let $\pi_M: C \rightarrow T^*M\setminus 0 $ and $\pi_\mathcal{M}: C \rightarrow T^*\mathcal{M}\setminus 0 $ be the natural projections.
    Then
    \begin{itemize}
        \item[(a)] $\pi_M: (y, \eta, \hat{y}, \hat{\eta}, x, \xi) \mapsto  (x, \xi) $ is a diffeomorphism,
        \item[(b)] $\pi_\mathcal{M}: (y, \eta, \hat{y}, \hat{\eta}, x, \xi) \mapsto  (y, \eta, \hat{y}, \hat{\eta})$ is a local diffeomorphism.
    \end{itemize}
\end{pp}
\begin{proof}
We prove (a) first. For each $(x,\xi)\in T^*M\setminus 0$, with
the property (A1),  there is a unique curve $\gamma$ passing $x$ at time $t$ and conormal to $\xi$.
Suppose $\gamma$ hits $H$ transversally at $y$ in the direction parameterized by $\eta$, at $t = 0$.
Then $(y,\eta)$ can be derived from the flow $\phi_{-t}(x,v)$, composed with 
the restriction on $H$, where $v$ is the unit vector corresponds to $\xi$.
Thus, in a small neighborhood $(y,\eta)$ depends on $(x,\xi)$ smoothly and  $(\hat{y},\hat{\eta})=(\mu a_1(t,y,\eta), \mu b_1(t,y,\eta))$ with $\mu = |\xi|$ also depends on $(x,\xi)$ in a smooth way, if we are away from the zero section. This implies one can regard $(x, \xi)$ as a parameterization of $C$ and $\pi_M$ is a diffeomorphism.

For (b), in local parameterization we have $\pi_\mathcal{M}: (y, \eta, t, \mu) \mapsto (y, \eta, \mu a_1, \mu b_1)$.
We compute $\diff \pi_\mathcal{M}$ to have
\[
\det \diff \pi_\mathcal{M} = \mu (\dot{a}_1 b_1 - \dot{b}_1 a_1) \equiv \mu W(t),
\]
where $W(t) $ is the Wronskian and it satisfies $\dot{W}(t) = V(\lambda) W(t)$.
Indeed,
recall $a_1, b_1$ are the projections of Jacobi fields onto $\dot{\gamma}^\perp(t)$
and therefore both of they should satisfy the equation in (\ref{Jeq_y}).
Solving the equation of the Wronskian, we have $W(t) = W(0) e^{\int V(\lambda) \diff t}$,
which indicates $W(t)$ is zero at any $t$, if it vanishes at some point.
When $W(t)$ vanishes, we have $a_1(t) = c b_1(t)$ for some constant $c$.
However, since
$a_1(0) \neq 0$ and
$b_1(0) = 0$ by the initial condition (\ref{eq_initial}),
it is impossible
and therefore we have $W(t) \neq 0$.
\end{proof}
We say the integral transform $I_w$ satisfies the Bolker condition, see \cite{Guillemin1979}, if the projection $\pi_\mathcal{M}$ is an injective immersion. 
\begin{thm}\label{thm_C_diff}
    The map $\mathcal{C}(x,\xi) \equiv \pi_\mathcal{M} \circ \pi_M^{-1}(x,\xi)$ 
    is a local diffeomorphism from $T^*M\setminus 0$ to $T^*\mathcal{M}\setminus 0$.
    Moreover, we have $\mathcal{C}(p_1, \xi^1) = \mathcal{C}(p_2, \xi^2)$ if and only if there is a curve $\gamma(t,y,\eta)$ joining $p_1$ at $t_1$ and $p_2$ at $t_2$, with $t_2 > t_1$, such that
    \begin{itemize}
        \item[(a)] $p_2$ is conjugate to $p_1$.
        \item[(b)] $\xi^1 = \mu a_1(t_2) \dot{\gamma}^\perp (t_1)$ and $\xi^2 = \mu a_1(t_1) \dot{\gamma}^\perp (t_2)$ for some nonzero $\mu$, where $a_1(t)$ is defined in (\ref{jeq_ab}).
    \end{itemize}
In particular, if there are no conjugate points, then the Bolker condition is satisfied and $\mathcal{C}$ is a diffeomorphism.
\end{thm}
\begin{proof}
If there exist different $(p_1,\xi^1)$ at $t_1$ and  $(p_2,\xi^2)$ at $t_2$
such that $\mathcal{C}(p_1, \xi^1) = \mathcal{C}(p_2, \xi^2) = (y,\eta, \hat{y}, \hat{\eta})$,
we must have $\hat{y} b_1(t) - \hat{\eta}a_1(t) =0$ is true for both $ t= t_1, t_2$.
Consider the vector field
\[
c_0(y,\eta,t)
= \hat{y} \frac{\partial \gamma}{\partial \eta} - \hat{\eta} \frac{\partial \gamma}{\partial \eta}
= (\hat{y} b_1(t) - \hat{\eta}a_1(t)) e_1(t) + (\hat{y} b_2(t) - \hat{\eta}a_2(t))e_2(t).
\]
It satisfies
\begin{align*}
&c_0(y,\eta,t_1) = (\hat{y} b_2(t_1) - \hat{\eta}a_2(t_1))e_2(t_1) \equiv \mu_1 \dot{\gamma}(t_1).\\
&c_0(y,\eta,t_2) = (\hat{y} b_2(t_2) - \hat{\eta}a_2(t_2))e_2(t_2) \equiv \mu_2 \dot{\gamma}(t_2).
\end{align*}
It follows that we can define the Jacobi field
\begin{equation*}\label{jacobifield}
c(y,\eta,t)
= \hat{y} \frac{\partial \gamma}{\partial \eta} - \hat{\eta} \frac{\partial \gamma}{\partial \eta} - (\mu_1 \frac{t - t_2}{t_1 -t_2} +\mu_2 \frac{t - t_1}{t_2 -t_1} ) \dot{\gamma}(t),
\end{equation*}
which satisfies $c(y,\eta,t_1)=c(y,\eta,t_2)=0$. This implies $p_2$ is conjugate to $p_1$.
Suppose
\[
\xi^1 = \mu_1 \dot{\gamma}^\perp (t_1), \quad \xi^2 = \mu_2 \dot{\gamma}^\perp (t_2)
\]
By (\ref{conditioncanonical}), they satisfy
\begin{equation*}
\hat{y} = \mu_1 a_1(t_1) = \mu_2 a_1(t_2).
\end{equation*}
This indicates we can find nonzero $\mu$ such that $\mu_1 = \mu a_1(t_2)$ and $\mu_2 = \mu a_1(t_1)$.

Conversely, if $p_2$ is conjugate to $p_1$, then there is a nonzero Jacobi field $J(t)$
so that $J(t_1)=J(t_2)=0$.
More precisely, we have
\[
\kappa_1 \frac{\partial \gamma}{\partial y} + \kappa_2 \frac{\partial \gamma}{\partial \eta} + c_1 \dot{\gamma}(t) +c_2 t\dot{\gamma}(t) = 0,
\quad  \text{for } t = t_1, t_2.
\]
The projection on $\dot{\gamma}^\perp(t)$ shows
\begin{equation}\label{singularmatrix}
\begin{cases}
\kappa_1 a_1(t_1) + \kappa_2 b_1(t_1) =0\\
\kappa_1 a_1(t_2) + \kappa_2 b_1(t_2) =0
\end{cases}
\iff
\left[
\begin{array}{cc}
a_1(t_1) & b_1(t_1)\\
a_1(t_2) & b_1(t_2)
\end{array}
\right]
\left[
\begin{array}{cc}
\kappa_1\\
\kappa_2
\end{array}
\right]
= 0.
\end{equation}
Notice $\kappa_1, \kappa_2$ cannot be both zero, otherwise we have the trivial Jocabi field tangent to $\gamma$.
This implies the matrix in (\ref{singularmatrix}) is singular.
Now since $\xi^1 = \mu a_1(t_2) \dot{\gamma}^\perp (t_1)$ and $\xi^2 = \mu a_1(t_1) \dot{\gamma}^\perp (t_2)$ for some nonzero $\mu$,
by (\ref{conditioncanonical}) one have
\[
\hat{y} = \mu a_1(t_2) a_1(t_1) = \mu a_1(t_1)a_1(t_2), \quad
\hat{\eta} = \mu a_1(t_2) b_1(t_1) = \mu a_1(t_1)b_1(t_2),
\]
where the last equality is from the zero determinant of the matrix in (\ref{singularmatrix}).
This proves $\mathcal{C}(p_1, \xi^1) = \mathcal{C}(p_2, \xi^2)$.
\end{proof}
For $(p_1, \xi^1)$ and $(p_2, \xi^2)$ satisfying Theorem \ref{thm_C_diff}, we call them the \textit{conjugate covectors}.
\begin{thm}[{\cite[Proposition 2]{Frigyik2007}}]\label{thm_N}
If there are no conjugate points, the normal operator $N = I_w^* I_w$ is a $\Psi$DO of order $-1$ with principal symbol 
\[
\sigma_p(N)(x, \xi) = \frac{2\pi}{|\xi|}(|w(x, \xi_\perp)|^2 +|w(x, -\xi_\perp)|^2 ),
\]
where $\xi_\perp$ is the vector conormal to the covector $\xi$ with the same length and the rotation of $\pi/2$. 
\end{thm}
\begin{proof}
This theorem is proved in {\cite[Section 4]{Frigyik2007}}. By the proof of {\cite[Proposition 2]{Frigyik2007}}, the normal operator is a $\Psi$DO with the principal symbol
\[
\sigma_p(N)(x,\xi) = 2\pi\int_{S^1} A(x,0,\omega)\delta(\xi(\omega)) \diff \omega,
\]
where $\delta$ is the Dirac delta function and in local coordinates we have
\[
A(x, 0, \omega) = J^{-1}(x, 0, \omega)|w(x, \omega)|^2.
\]
Here we write 
\[
\exp_x(t,v) - x = t m(t,v;x)
\]
and introduce the change of variables $(r, w) \in \mathbb{R} \times S^1$ by 
\[
r = t|m(t,v;x)|, \quad \omega = m(t,v;x)/|m(t,v;x)|,
\]
with the Jacobian 
\[
J(x, t, v) \equiv \det \frac{\partial (r, \omega)}{\partial (t, v)}. 
\]
For more details, see equations (32) - (34) in \cite{Frigyik2007}.
Note that
\[
m(0,v;x) = \dot{\gamma}_{x,v}(0) = v
\] since we use arc length parameterization. 
Thus, we have that $J(x, 0, v) = 1$ and the principal symbol of $N$ is given as claimed. 
See also \cite[Theorem 5.2]{MR2970707} and its proof. 
\end{proof}

\subsection{The local problem}
In this following, we present the cancellation of singularities arising in the local inverse problem for the integral transform $I_w$,  if there are conjugate covectors.
This is the analog to the case of geodesic ray transforms in \cite{Monard2015}.

Consider a fixed curve $\gamma \in \Gamma$ with conjugate covectors $(x_1, \xi^1)$ at $t_1$ and $(x_2, \xi^2)$ at $t_2$.
Let $V^j$ be small conic neighborhoods of $(x_j, \xi^j)$, with base $U_j$ as a small neighborhood of $x_j$, for $j=1,2$.
With $\mathcal{C}$ being a local diffeomorphism, it maps a small conic neighborhood of $(x_j, \xi^j)$ to one of $\mathcal{C}(x_j, \xi^j)$, for $j = 1, 2$.
Notice $V^j$ should have two disjoint components, for more details see \cite{Monard2015}.
By shrinking those neighborhoods a bit, one can assume that $\mathcal{C}(V^1) = \mathcal{C}(V^1) \equiv \mathcal{V}$.
We define the restriction $\mathcal{C}_j \equiv \mathcal{C}|_{V^j}$ for $j=1,2$.
Note that $\mathcal{C}_j$ are diffeomorphisms.
It follows that
\[
\mathcal{C}_{12} \equiv \mathcal{C}_1^{-1}  \mathcal{C}_2: V^2 \rightarrow V^1, \quad \quad \mathcal{C}_{21} \equiv \mathcal{C}_2^{-1}  \mathcal{C}_1: V^1 \rightarrow V^2
\]
are also diffeomorphisms.
Let $I_j$ be $I_w$ restricted to distributions with wave front sets supported in $V^j$, for $j =1,2$.
Then $I_j$ are FIOs with canonical relations ${C}_j$, where $C_j$ are restriction of the canonical relation $C$ to $V^j \times \mathcal{V}$. 
When the weights are nonvanishing, the restriction $I_j$ are elliptic FIOs, and therefore we can define
\begin{align}\label{def_F12}
F_{21} = I_2^{-1}I_1, \quad F_{12} = I_1^{-1}I_2.
\end{align}
Note that they are FIOs with canonical relations 
$C_2^{-1}\circ C_1$ and $C_1^{-1}\circ C_2$ respectively.
We can show the following result by the same arguments in \cite{Monard2015}.
\begin{thm}\label{thm_microkernel}
Suppose $w(x_1, \xi^1_\perp), w(x_2, \xi^2_\perp) \neq 0$.
Let $f_j \in \mathcal{E}'(U_j)$ with $\text{\( \WF \)}(f_j) \subset V^j$, for $j = 1,2$.
Then the local data
    \[
    {{I_w}} (f_1+f_2) \in H^s(\mathcal{V})
    \]
    if and only if
    \[
    f_1 + F_{12} f_2 \in H^{s-1/2}(V^1) \Leftrightarrow
    F_{21} f_1 +  f_2 \in H^{s-1/2}(V^2),
    \]
    where $F_{12}$ and $F_{21}$ are elliptic FIOs defined in (\ref{def_F12}). 
\end{thm}
This theorem indicates that
given a distribution $f_1$ singular in $ V^1$, there exists a distribution $f_2 = -F_{21}f_1$ singular in $V^2$ such that the transform is smooth.
In other words, if we suppose $f=f_1$, the singularities of $f$ cannot be resolved from the singularities of the transform $I_w f$.
Indeed, the singularities of $f$ can only be recovered up to an error in the microlocal kernel, i.e., an error in form of $(\text{Id} - F_{21})h_1$ with some $h_1$ singular in $V^1$, 
since $I_w ((\text{Id} - F_{21})h_1)$ is always smooth.
For a more detailed description, see \cite{Holman2017}.

\subsection{Artifacts}\label{subsec_artifacts}
In this subsection, we describe the artifacts arising in the reconstruction from the local data, when there are conjugate points.
For convenience, we assume the weight $w = 1$ and use the notation $I$ instead of $I_w$ in the following.

First, we consider the backprojection ${{I}}^*{{I}}$ to reconstruct $f$, in the presence of conjugate covectors. 
Suppose $\gamma$ is the $\lambda$-geodesic in Theorem \ref{thm_C_diff} with conjugate covectors $(x, \xi)$ in $V^1$ and $(y, \eta)$ in $V^2$.
Let $f = f_1 + f_2$ with $f_j$ singular in $V^j$, $j = 1,2 $. 
In a small neighborhood of $\gamma$, we have
\begin{equation*}
{{I}}^*{{I}}f = I_1^*I_1 f_1+I_1^*I_2 f_2+I_2^*I_1 f_1+I_2^*I_2 f_2.
\end{equation*}
Recall $I_1$ and $I_2$ are defined microlocally and are elliptic FIOs of order $-1$ with canonical relations ${C}_1$ and ${C}_2$, which are diffeomorphisms.
Then by Theorem \ref{thm_N}, 
$I_1^*I_1$ and $I_2^*I_2$ are elliptic $\Psi$DOs of order $-1$ with principal symbol
 ${4\pi}/{|\xi|}$.  
For $I_1^* I_2$ and $I_2^*I_1$, 
by \cite[Theorem 25.2.2]{Hoermander2009} and the transversal composition calculus,
they are FIOs of order $-1$ associated with canonical relations $C^{-1}_1 \circ C_2 $ and $C^{-1}_2 \circ C_1$ respectively.

Let  $-\Delta_g$ be the Laplacian operator in $M$ and its square root $\sqrt{-\Delta_g}$ is a $\Psi$DO of order $1$ with principal symbol $|\xi|_g$. 
Let $\Lambda = \sqrt{-\Delta_g}/{(4\pi)}$. 
Then module lower order operators, one has
\[
\Lambda I_1^*I_1 \equiv \text{Id}, \quad \Lambda I_1^*I_1 \equiv \text{Id}, 
\quad \Lambda I_2^*I_1 \equiv F_{21}, \quad \Lambda I_1^*I_2 \equiv F_{12}.
\]
It follows that 
\begin{align}\label{eq_fb}
\Lambda I^*If = f_1 + F_{21} f_1 + f_2 + F_{12} f_2
\end{align}
up to lower order terms.
This implies that we recover the singularities of $f_1 + f_2$ together with $ F_{12} f_1 + F_{21} f_2$ from the backprojection.
The later are artifacts. 
If we write a distribution singular in $V^1 \cup V^2$ as a vector-valued function with the first component equal to its restriction to $V^1$ and the second component equal to its restriction to $V^2$, then 
\begin{equation*}
\Lambda{{I}}^*{{I}}f = 
\begin{bmatrix}
\begin{array}{ll}
\text{Id} &  F_{12}\\	
F_{21}&  \text{Id}
\end{array}
\end{bmatrix}
\begin{bmatrix}
\begin{array}{l}
f_1\\	
f_2
\end{array}
\end{bmatrix}
\equiv M 
\begin{bmatrix}
\begin{array}{l}
f_1\\	
f_2
\end{array}
\end{bmatrix}
\end{equation*}
up to lower-order terms. 
Especially when $f_2 = 0$, from the filtered backprojection $\Lambda{{I}}^*{{I}}f$, we recover $f_1 + F_{21}f_1$ and therefore the artifacts equal to $F_{21}f_1$ arises in the reconstruction. 

Next, we consider the numerical reconstruction by using the Landweber iteration as in \cite{Holman2017}. 
For more details of the method, see \cite{STEFANOV2017}.
We follow the same argument in \cite{Monard2015,Holman2017},
see also \cite{Zhang2020}.
Let $\chi$ be a smooth cutoff in $\mathbb{R}^2$ with $\chi = 1$ in a small neighborhood of $M$ to avoid dealing with non-local operators. 
We set $\mathcal{L} = \Lambda \chi I^*I$ and compute 
\[
\mathcal{L}^* \mathcal{L} =I^*I \chi \frac{1}{4\pi}({-\Delta_g})\chi I^*I
= 
\begin{bmatrix}
\begin{array}{ll}
\text{Id} + F^*_{21}F_{21}&  F^*_{21} + F_{12}\\	
F_{12}^* + F_{21}&  \text{Id} + F_{12}^*F_{12}
\end{array}
\end{bmatrix}.
\]

Let $g$ be the local data and it is assumed be in the range of ${{I}}$.
Now we use the Landweber iteration to solve the equation ${{I}} f = g$. 
We write
\begin{equation}\label{landweber}
(\text{Id} - (\text{Id}- \gamma \mathcal{L}^* \mathcal{L})) f = \gamma \mathcal{L}^*\Lambda\chi I^* g.
\end{equation}
Then with a small enough and suitable $\gamma > 0$, it can be solved by the Neumann series and we have the truncated scheme
\[
f^{(N)} = \sum_{k=0}^{N}(\text{Id}- \gamma \mathcal{L}^*\mathcal{L})^k 
\gamma \mathcal{L}^*\Lambda\chi I^* g.
\]
This series converge to the minimal norm solution to $\mathcal{L} f = \Lambda\chi I^* g$.
Suppose the original function is $f = f_1 + f_2$ with $f_2 = 0$. 
The analysis in \cite[Section 3.2.3]{Holman2017} shows that 
\begin{align*}
\text{Landweber solution} = 
[f_1 - (\text{Id} + F_{21}^*F_{21})^{-1}f_1] + [F_{21}(\text{Id} + F^*_{21}F_{21})^{-1}f_1],
\end{align*}
where the first square brackets are terms microlocally supported in $V^1$ and the second term in $V^2$. 
The artifacts arising in the reconstruction is 
\begin{align*}
f - \text{Landweber solution} = 
[(\text{Id} + F_{21}^*F_{21})^{-1}f_1] - [F_{21}(\text{Id} + F^*_{21}F_{21})^{-1}f_1].
\end{align*}

We emphasize that the artifacts above arises in the reconstruction from the local data.
If we consider the recovery from the global data, i.e., with the knowledge of the integral transform over all curves in $\Gamma$, 
then the singularities of $f$ might be recoverable.
This is because the singularities can be probed by more than one smooth curves in $\Gamma$. 
In some cases, the recovery of certain singularity depends on a discrete dynamical system, i.e., a sequence of conjugate covectors, inside $M$, see (\ref{eq_M}). 
If this sequence goes out of $M$, then we can resolve the corresponding singularity as is discussed in \cite{MR3080199, Zhang2020}.
For more details, see Proposition \ref{pp_m}.

\section{Examples of $\lambda$-geodesics with conjugate points}\label{sec_example}
In this section, we present several examples of the family of curves that we study.
These curves are different from geodesics and we show the conjugate points exist.
\begin{example}\label{ex_c}
The first example comes from \cite{MR2970707}.
Let $\Gamma_1$ consist all the unit circles in $\mathbb{R}^2$ with a fixed orientation.
These circles are actually the magnetic geodesics w.r.t the Euclidean metric and a constant non-zero magnetic field by \cite{MR2351370}.
Suppose they have a fixed orientation and we parameterize a unit circle through point $p =(p^1, p^2)$ in the direction of $v = (\cos \theta, \sin \theta)$ by
\[
\gamma_{p, \theta}(t) = (p^1, p^2) + (\cos(t + \theta), \sin (t+ \theta)) - (\cos \theta, \sin \theta).
\]
We have
\[
\dot{\gamma}_{p, \theta}(t) =  (-\sin(t + \theta), \cos (t+ \theta)), \quad  \ddot{\gamma}_{p, \theta}(t) = - (\cos(t + \theta), \sin (t+ \theta)),
\]
and therefore
\[
\ddot{\gamma}_{p, \theta}(t) =
\begin{pmatrix}
    0 & -1 \\
    1 & 0
\end{pmatrix}
\dot{\gamma}_{p, \theta}(t).
\]
One can check that the properties (A1) and (A2) are satisfied.
As is shown in \cite{MR2970707}, since
\begin{align}\label{eq_dexp}
\det \diff_{s, v} \exp_p(t,v) =\det \frac{\partial \gamma_{p, \theta}(t)}{\partial (t, \theta)} = \sin t,
\end{align}
for each point $p$, it has conjugate points corresponding to $t = \pi$ and any $\theta \in \mathbb{R}$.
Thus, the conjugate locus of $p$ is 
\[
\Sigma(p) = \{\gamma_{p, \theta}(\pi), \ \theta \in \mathbb{R}\} = \{y:\ |y-p| = 2\},
\]
which is the circle centered at $p$ with radius equal to $2$.

Now let $(p_1, \xi^1) \in T^*M$ 
and suppose $f$ is singular near $(p_1, \xi^1)$. 
We would like to find out the conjugate covector $(p_2, \xi^2)$ of $(p_1, \xi^1)$. 
To parameterize $\Gamma_1$, we assume $p_1$ is near the origin and $\xi^1$ is in a conic neighborhood of $\xi^0 = (1,0)$.
Consider the line $H = \{(y, 0), y \in \mathbb{R}\}$.
By the notations used in Section \ref{sec_microlocal}, a unit circle passing the point $(y, 0) \in \mathbb{R}^2$ in the direction of $v = (\cos \eta, \sin\eta)$ is now parameterized by 
\[
\gamma(t, y,\eta) = (y,0) + (\cos(t + \eta), \sin (t+ \eta)) - (\cos \eta, \sin \eta). 
\]

Suppose $p_1 = (y_1, 0)$ and 
$
\xi^1  = (\cos \eta_1, \sin \eta_1).
$
By Theorem \ref{thm_C_diff}, $(p_2, \xi^2)$ and $(p_1, \xi^1)$ are conjugate covectors when
$p_2$ is conjugate to $p_1$ and 
\[
\xi^1 = \mu a_1(t_2) \dot{\gamma}^\perp (t_1), \quad \xi^2 = \mu a_1(t_1) \dot{\gamma}^\perp (t_2)
\] for some nonzero $\mu$. 
Then by equation (\ref{eq_dexp}), we write  
\[
p_2 = \gamma(\pi, y_1, \eta_1), \quad \text{or } p_2 = \gamma(\pi, y_1, -\eta_1)
\]
with $t_1 = 0, t_2 = \pi$, which depends on the different sign of $\mu$. 
We denote the first case by $(p_2^{(1)}, \xi^2_{(1)})$ and the second case by $(p_2^{(2)}, \xi^2_{(2)})$. 
We compute 
\[
p^{(1)}_2 = (y_1, 0) - 2(\cos \eta_1, \sin \eta_1), \quad p^{(2)}_2 = (y_1, 0) + 2(\cos \eta_1, \sin \eta_1). 
\]
Since $a_1(t) = (\dot{\gamma}^\perp(t), \frac{\partial \gamma(t,y,\eta)}{\partial y})$, 
we can find
\[
\xi^2_{(1)} = \xi^1, \quad \xi^2_{(2)} = \xi^1. 
\]
Therefore, there are two covectors that are conjugate to $(p_1, \xi^1)$ given by 
\begin{equation}\label{eq_cc}
(p_2^{(1)}, \xi^2_{(1)}) = (p_1 - 2(\cos \eta_1, \sin \eta_1), \xi^1), \quad (p_2^{(2)}, \xi^2_{(2)}) = (p_1 + 2(\cos \eta_1, \sin \eta_1), \xi^1). 
\end{equation}

\begin{example}
In this example, let $\Gamma_2$ consist all identical ellipses with a fixed orientation starting from different points
\[
\gamma_{p}(t) = (p^1, p^2) + (a\cos t -a , b\sin t),
\]
where $a, b$ are constants. Different values of $a, b$ give us different $\Gamma_2$.
We have
\[
\dot{\gamma}_{p}(t) =  (-a\sin t, b\cos t), \quad \ddot{\gamma}_{p}(t) =  -(a\cos t, b\sin t).
\]
One can verify that $\Gamma$ satisfies the properties (A1) and (A2).
Indeed, first we have
\[
\ddot{\gamma}_{p}(t) =
\begin{pmatrix}
0 & -{a}/{b}\\
{b}/{a} & 0
\end{pmatrix}
\dot{\gamma}_{p}(t).
\]
Then, for any point $x = (x_1, x_2)$ and $v = (\cos \theta, \sin \theta)$,
there is a curve $\gamma_p$ passing $x$ in the direction of $v$ if we can find some $p_0, t_0$ such that
\[
-a \sin t_0 = \mu \cos \theta, \quad b \cos t_0 = \mu \sin \theta, \quad {\gamma}_{p}(t) =  x,
\]
for some $\mu >0$. This is true since we can solve $t_0$ from
\[
\cos t_0 = \frac{-\mu }{a} \cos \theta,
 \quad  \sin t_0 = \frac{\mu }{b}\sin \theta,
\]
where $ \mu  = \frac{ab}{\sqrt{b^2 \cos^2 \theta + a^2 \sin^2 \theta}}$ and then
$p_0= x - (a\cos t_0 -a, b \sin t_0) $.
One can see that $t_0$ is unique in the sense of modulo $2\pi$ and it depends on $\theta$ in a smooth way.
We denote it by $t_0(\theta)$ and compute
\begin{gather*}
\frac{\diff t_0}{\diff \theta} = \frac{b^2 \cos^2 \theta + a^2 \sin^2 \theta}{ab}. \\
\end{gather*}
To parametrize $\Gamma_2$, we pick $H = {(0, x^2)}$.
Notice for every point $x = (0, x^2) \in H$ and any direction $v = (\cos \theta, \sin \theta)$,
there is a unique $\gamma_p(t)$ passing $x$ in the direction of $v$.
From the analysis above, we reparameterize this ellipse as
\[
\rho_{x^2, \theta}(t) = (a\cos(t+t_0 )- a \cos t_0, b \sin(t + t_0)-b \sin t_0 + x^2).
\]
Thus, for any $x = (0, x^2)$ we have
\[
\det \diff_{s, v} \exp_x(t,v) =\det \frac{\partial \rho_{x^2, \theta}(t)}{\partial (t, \theta)} =  \frac{\diff t_0}{\diff \theta} ab \sin t,
\]
and its conjugate points are corresponding to $t = \pi$ and any $\theta$.
This gives us the conjugate locus $\Sigma(x) = \{\rho_{x^2, \theta}(\pi), \ \theta \in \mathbb{R}\} = \{x - 2(a\cos t_0, b\sin t_0), t_0 \in \mathbb{R}\}$,
which is a larger ellipse centered at $x$.
\end{example}
\begin{example}
In this example, we choose the a family of smooth curves that are locally defined.
For $p = (p^1, p^2)$, define
\[
\gamma_p(t) = e^{at} (\cos(t - \varphi) - \cos (-\varphi), \sin (t - \varphi) - \sin (-\varphi)) + p,
\]
where $a$ is a nonzero constant and $\varphi$ satisfies 
\[
\cos \varphi = \frac{a}{\sqrt{a^2 + 1}}, \quad \sin \varphi =\frac{1}{\sqrt{a^2 + 1}}.
\]
We compute
\[
\dot{\gamma}_p(t) = e^{at} \sqrt{a^2 + 1}(\cos t, \sin t),  \quad \ddot{\gamma}_p(t) = e^{at} \sqrt{a^2 + 1}(a\cos t - \sin t , a\sin t + \cos t),
\]
and one can show that
\[
\ddot{\gamma}_{p}(t) =
\begin{pmatrix}
    a & -1\\
    1 & a
\end{pmatrix}
\dot{\gamma}_{p}(t),
\]
which implies the property (A2).
To satisfy the property (A1) at least locally,
we fix a curve
\[
\gamma_0(t) = e^{at} (\cos(t - \varphi), \sin (t - \varphi)),\text{ for  } t \in (-\frac{\pi}{4}, \frac{7 \pi}{4})
\]
which passes the point $(0,0)$ at $t=0$ in the direction $\theta = 0$.
Now let
\[
\Gamma_3 = \{ \gamma_0(t) + p, \text{ where } t \in (-\frac{\pi}{4}, \frac{7 \pi}{4}) \}.
\]
Then $\Gamma_3$ is an open family of smooth curves.
For each point $x = (x_1, x_2)$ and direction $v = (\cos \theta, \sin \theta)$ with $\theta \neq -\frac{\pi}{4}$,
there is a unique
\[
t_0 = \theta, \quad p_0 = x - e^{a \theta}(\cos(\theta - \varphi), \sin(\theta - \varphi) + (\cos \varphi, -\sin \varphi),
\]
such that $\gamma_{p_0}(t_0)$ through $x$ in the direction of $v$.
We reparameterize $\Gamma_3$ using $H = {(0, x^2)}$ by writing these curves as
\begin{align*}
\rho_{x^2, \theta}(t) &= e^{a(t+\theta)} (\cos(t +\theta - \varphi) , \sin (t +\theta - \varphi) )\\
 &- e^{a\theta}(\cos (\theta-\varphi), \sin (\theta-\varphi))) + (0,x^2),
\text{ for  } t \in (-\frac{\pi}{4}, \frac{7 \pi}{4}).
\end{align*}
Notice that $\rho_{x^2, \theta}(t)$ passes $(0,x^2)$ at $t=0$ in the direction of $\theta$.
We mention that $\Gamma_3$ forms a neighborhood of $\rho_{0, 0}(t)$
and we compute
\[
\det \diff_{s, v} \exp_x(t,v) =\det \frac{\partial \rho_{x^2, \theta}(t)}{\partial (t, \theta)} =  e^{2a + \theta}\sin t.
\]
Thus, for the point $x = (0 ,0)$, there are conjugate points corresponding to $t = \pm \pi$ and $\theta \in (-{\pi}/{4}, {7 \pi}/{4})$.
The conjugate locus is
\begin{align*}
\Sigma(x) &= \{\rho_{0, \theta}(\pm\pi), \ \theta \in \mathbb{R}\} \\
&= \{ - (e^{\pm a\pi} + 1)e^{a\theta}(\cos (\theta - \varphi), \sin (\theta - \varphi)),
\text{ for  } \theta \in (-\frac{\pi}{4}, \frac{7 \pi}{4}) \}.
\end{align*}
In Figure \ref{picref}, we choose $a = -1/4$ (this implies $\varphi = \arccos(1/\sqrt{17})$) and we plot all $\rho_{0, \theta}(t)$ for $\theta \in (-{\pi}/{4}, {\pi}/{4})$. The light green curve is part of the conjugate locus $\Sigma(x)$.
\begin{figure}[h]
    \includegraphics[height=0.3\textwidth]{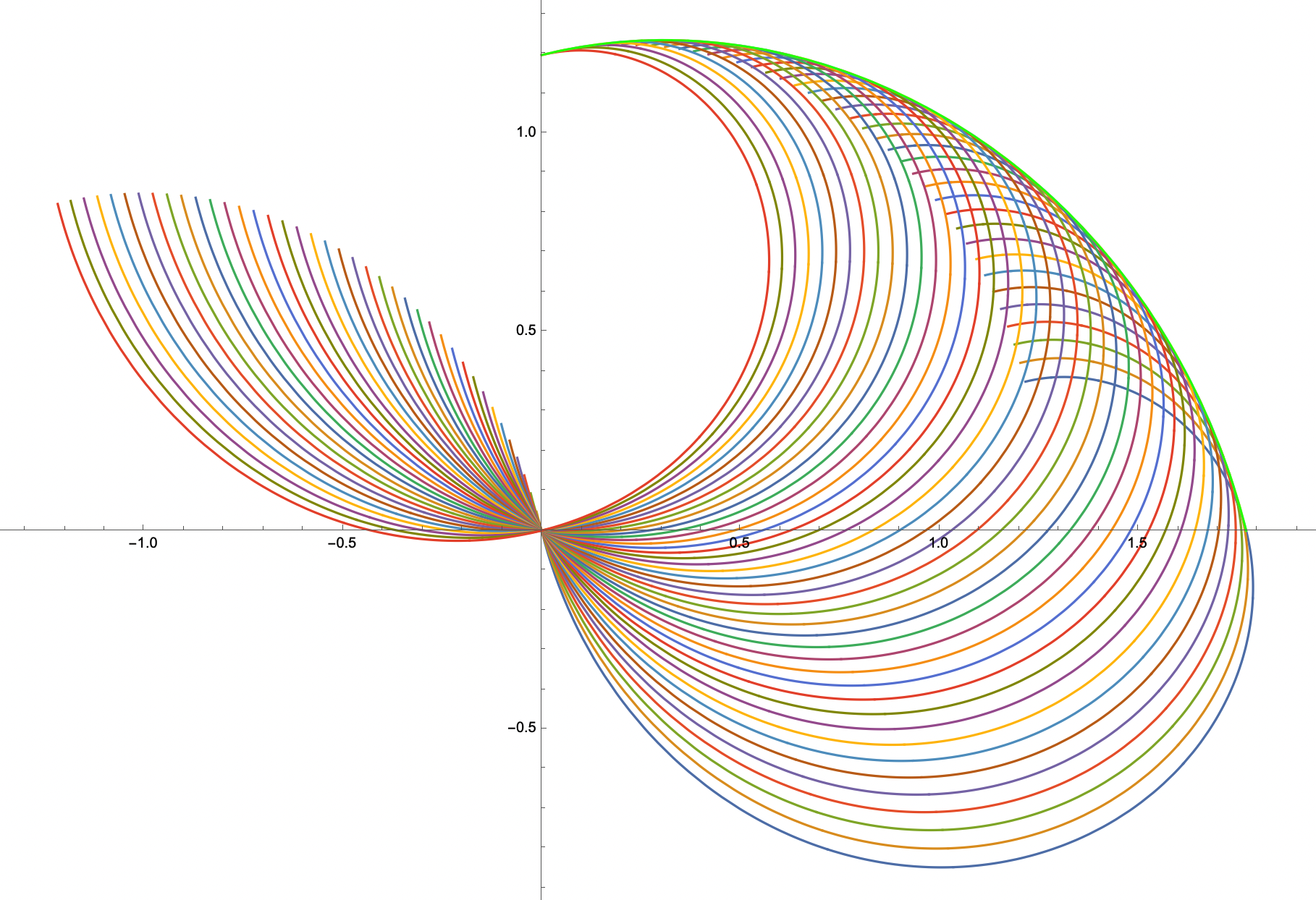}
    \caption{ Light green curve: part of conjugate locus $\Sigma(x)$. }
    \label{picref}
\end{figure}

\end{example}

\end{example}
\section{Numerical experiments}\label{sec_num}
This section aims to illustrate the artifacts arising
in the reconstruction by numerical experiments. 
We consider the family of unit circles with a fixed orientation, 
i.e., the magnetic geodesics w.r.t. the Euclidean metric and a constant nonzero magnetic field. 
See Example \ref{ex_c} in Section \ref{sec_example} and \cite{MR2970707} for more details.

More explicitly, let $M$ be a bounded domain without boundary in $\mathbb{R}^2$, for example, 
the open disk of radius $R=3$ centered at the origin.
Suppose $f$ is a smooth function supported in $M$. We define 
$$
I f(x_1,x_2) = \int_{|w|=1} f((x_1,x_2) + w) d l_w = \int_0^{2\pi} f(x_1+\cos \alpha, x_2 + \sin \alpha) d\alpha
$$	
as the integral transform performed over unit circles with radius $1$. 
These circles can be parameterized by their centers. 
Then the adjoint operator is
$$
I^* g(y_1,y_2) = \int_{|w|=1} g((y_1+y_2) - w) d l_w = \int_0^{2\pi} g(y_1-\cos \alpha, y_2 - \sin \alpha) d\alpha
$$
which coincides with $I$ itself. 
Numerically we compute $I f$ by the following steps.
\begin{itemize}
    \item[(1)] Discretization. We introduce $6 N \times 6N$ equispaced points in the square domain $[-3,3]\times[-3,3]$ so the grid spacing is ${1}/{N}$. 
    We discretize the input function $f$ over the grids. 
    \item[(2)] We compute the integral by the Trapezoidal method
    $$
    I f(x_1,x_2) = \sum_{k=1}^{n} \frac{f(z_{k-1},w_{k-1}) + f(z_{k},w_{k})}{2} \Delta \alpha + O(\frac{1}{n^2}),
    $$
    where $z_{k} = x_1 + \cos{{2 \pi k}/{n}}, w_{k} = x_2 + \sin{{2 \pi k}/{n}}$ and $\Delta \alpha = \frac{2 \pi}{n}$.
    When the input function is only given over grids, we can use the bilinear interpolation to approximate its values at $\{(z_k,w_k), k = 1, \ldots, n\}$ in Step (2) before the numerical integration.
\end{itemize}
Then we use the same method to compute $I^*$. 
\subsection{Backprojection}
First, we consider the reconstruction of $f$ from the transform by the backprojection $I^*If$.
To check the numerical implementation of the backprojection, one can compare the numerical result with the analytical one given by the formula 
\[
I^*If(x_1,x_2) = \int_{0}^{2} \int_{0}^{2\pi} \frac{4}{4 - r^2} f(x_1+r \cos \theta, x_2+ r\sin \theta) d\theta \diff r.
\]
in \cite{MR2970707}. 
Let $f_1$ 
be a truncated Gaussian concentrated near the origin, as an approximation of a delta function at the origin, and let $f_2 = 0$. 
We choose $N = 40$ and  the relative error between the numerical result $R_N$ and the analytical one
\[
\frac{\|{R_N - I^*If }\|_2}{\|I^*If\|_2} = 0.0140
\]
is relatively small.
In this case, $f$ is discretized as a matrix of $240\times240$ and we integrate it over unit circles. 
From Figure \ref{fig_1}, we can see
the artifacts appear exactly in the location of conjugate points. 
Indeed, by the analysis in Example \ref{ex_c}, the conjugate locus of the origin is the circle centered at the origin of radius $2$.
With $f$ as an approximation of the delta function, the singularities of $f$ are located near the origin in all directions. 
The conjugate covectors of the singularities of the delta function are described by the equation (\ref{eq_cc}), which are covectors conormal to the circle of radius $2$. 
Then by (\ref{eq_fb}), 
from the backprojection we recover both the singularities of $f = f_1$ and the singularities of $ F_{12} f_1$.
Note here we can omit $\Lambda$ in (\ref{eq_fb}) if we only consider the singularities, since $\Lambda$ is an elliptic $\Psi$DO. 

\begin{figure}[h]
    \centering
    \begin{subfigure}{0.5\textwidth}
        \centering
        \includegraphics[width=0.8\linewidth]{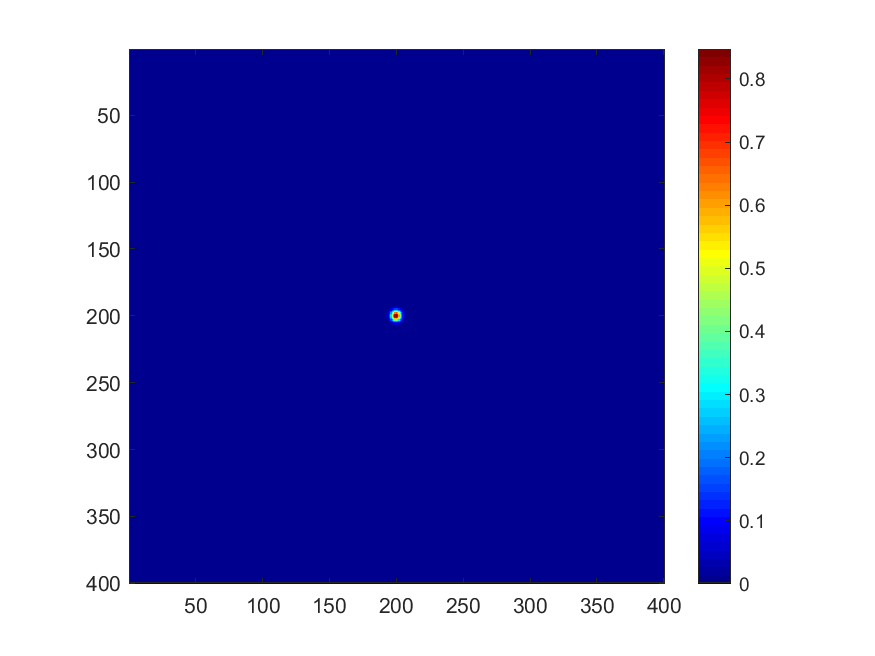}
    \end{subfigure}%
    \begin{subfigure}{0.5\textwidth}
        \centering
        \includegraphics[width=0.8\linewidth]{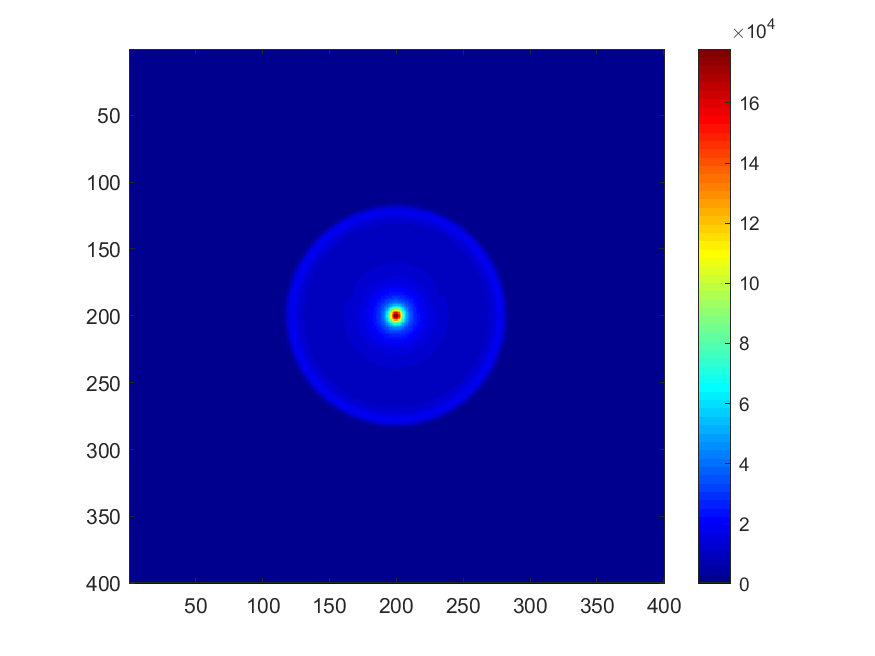}
    \end{subfigure}	
    \caption{The true $f$ and the reconstruction from the backprojection.}\label{fig_1}
\end{figure}
\subsection{Landweber iteration}
In the following, we choose $f = f_1$ to be a modified Gaussian with singularities located both in
certain space and in certain direction, that is, a coherent state, as is shown in Figure \ref{fig_2} (a).
Note the singularities of $f$ are actually semi-classical and they are located near the origin. 
The same analysis works (see \cite{Stefanov2020} for FIOs and semi-classical wave front sets) and here we use this coherent state to illustrate the artifacts in the analysis before. 
If we use the backprojection to reconstruct $f$, the artifacts appear in the location of conjugate points, 
see Figure \ref{fig_2} (b).  
This is described by (\ref{eq_cc}) and (\ref{eq_fb}). 

Next, we use the Landweber iteration to reconstruct $f$.
The analysis in Section \ref{subsec_artifacts} explains the artifacts in the reconstruction from the  local data.
However, if we use the global data and have the prior knowledge that $f$ is supported in a compact set, then we can recover the singularities of $f$ without artifacts.
The following proposition is an analog to \cite[Corrollary 3]{Zhang2020}. 
\begin{pp}\label{pp_m}
    Suppose $f \in \mathcal{E}'(\mathbb{R}^2)$ and $If \in C^\infty(\mathcal{M})$. Then $f$ is smooth. 
\end{pp}
\begin{proof}
Let $(x_0,\xi^0) \in T^*M$ and assume it is in the wave front set of $f$. 
This singularity can be canceled by its conjugate covectors, if they exist. 
As in \cite[Section 5]{Zhang2020}, we define
\begin{align}\label{eq_M}
\mathcal{M}(x_0,\xi^0) = \{ (x_k,\xi^k), 
&\text{ if it exists and is conjugate to $(x_{k'},\xi^{k'})$}, \\
& \text{where } k'= k-\operatorname{sgn} k, \text{for }k = \pm 1, \pm 2, \ldots \} \nonumber
\end{align}
as the set of all conjugate covectors related to $(x_0,\xi^0)$.
We can assume $\xi^0 = (\cos \eta_0, \sin \eta_0)$ is in a conic neighborhood of the covector $\xi = (1,0)$.
By (\ref{eq_cc}), we have 
\begin{align*} 
\mathcal{M}(x_0,\xi^0) = \{ (x_k,\xi^0): x_k = x_0 + 2k(\cos \eta_0, \sin \eta_0), \text{ for } k = \pm 1, \pm 2, \ldots 
\}.
\end{align*}
Let $V^k$ be a small conic neighborhoods of $(x_k ,\xi^k) \in \mathcal{M}(x_0,\xi^0)$.
Let $f_k$ be $f$ microlocally restricted to $V^k$ and $I_k$ be $I$ restricted to distributions singular in $V^k$. 
For each $k$, near the curve $\gamma_k$ where  $(x_{k-1} ,\xi^{k-1})$ and  $(x_k ,\xi^k)$ are conjugate covectors, 
we have the equation of cancellation of singularities 
\begin{equation}\label{eq_cs}
I_{k-1}f_{k-1} + I_k f_k = 0 \mod C^\infty,
\end{equation}
if we shrink $V^k$ such that $C(V^k)= V^{k-1}$. 
Since $f$ has compact support, there exist $k_+>0$ and $k_-<0$ such that $f$ is smooth near all $(x_k, \xi^k)$ with $k < k_+$ or $k< k_-$. 
Then we have 
\[
I_k f_k = 0 \mod C^\infty, \text{ for any } k < k_+ \text{ or } k< k_-.
\]
By (\ref{eq_cs}), it follows that $I_k f_k = 0 \mod C^\infty$ for all $k$. 
Thus, we have $f$ is smooth. 
\end{proof}
\begin{figure}[h]
    \centering
    \begin{subfigure}{0.3\textwidth}
        \centering
        \includegraphics[width=0.9\linewidth]{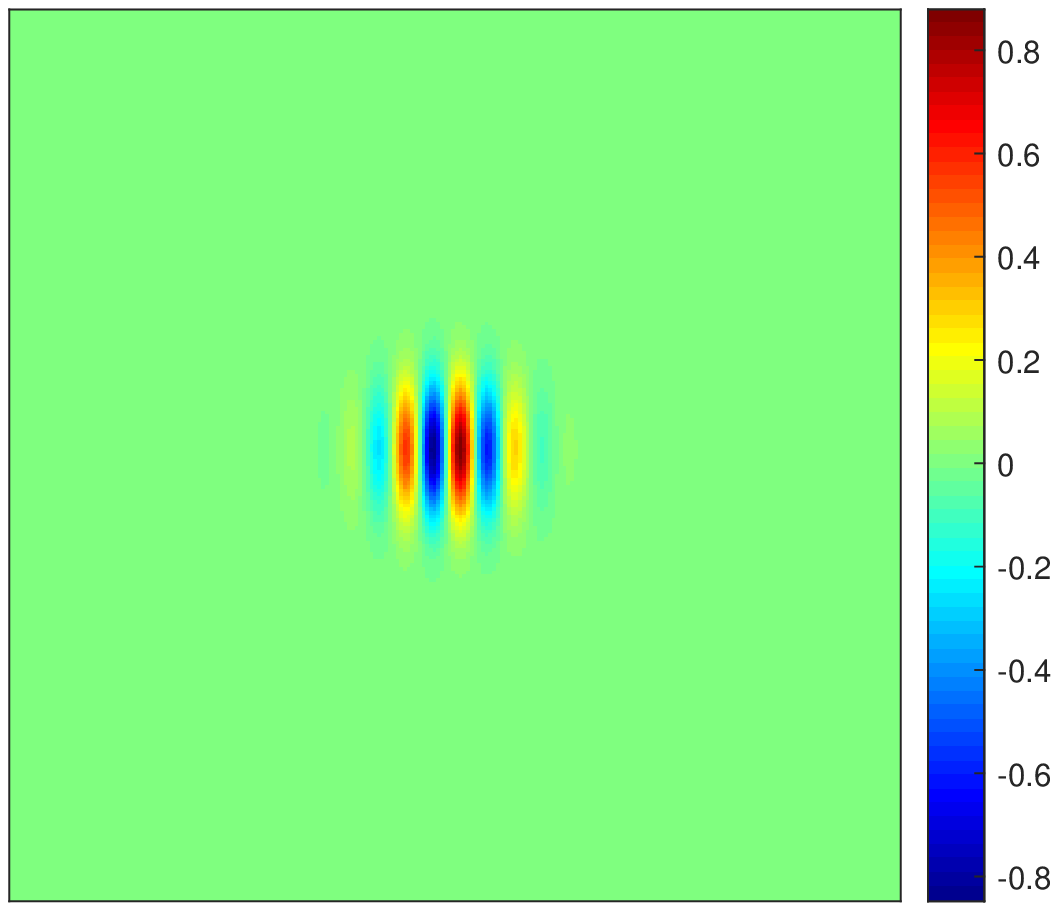}
        \subcaption{true $f$}
    \end{subfigure}%
    \begin{subfigure}{0.3\textwidth}
        \centering
        \includegraphics[width=0.9\linewidth]{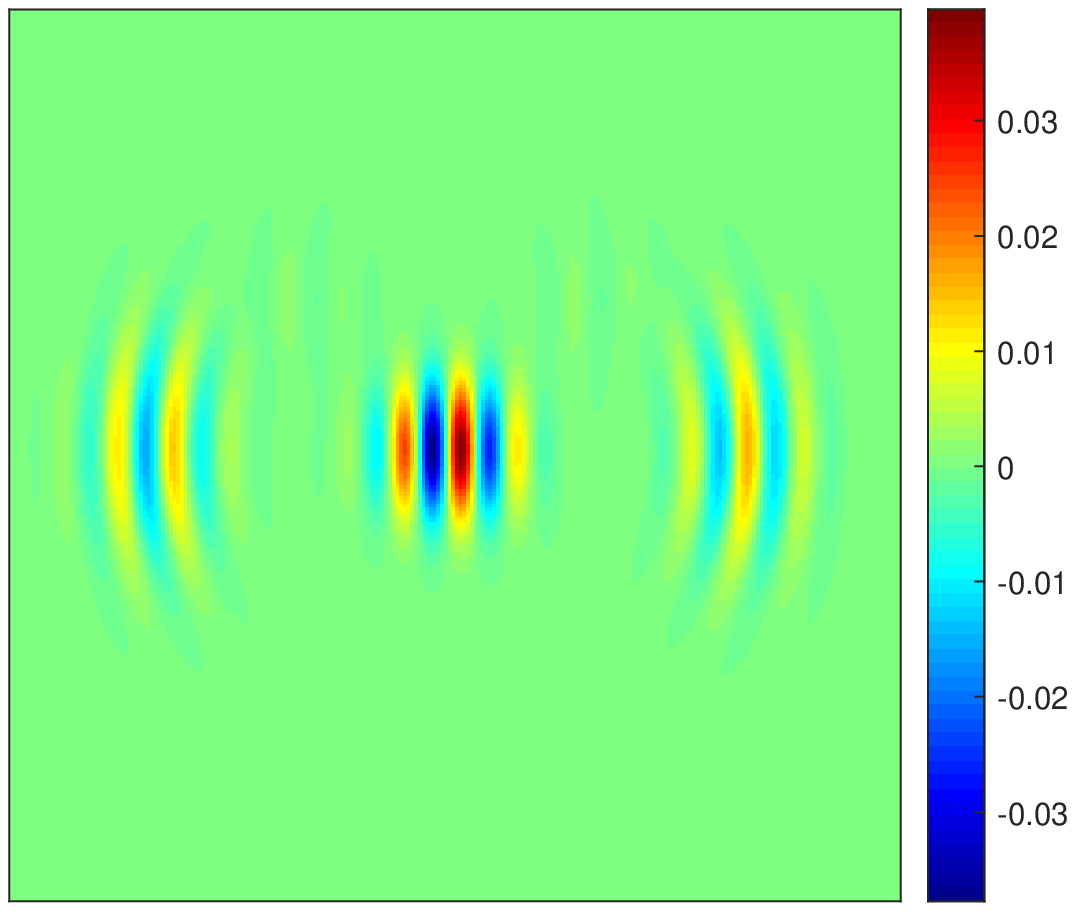}
        \subcaption{backprojection $f^{(1)}$}
    \end{subfigure}
    \begin{subfigure}{0.3\textwidth}
        \centering
        \includegraphics[width=0.9\linewidth]{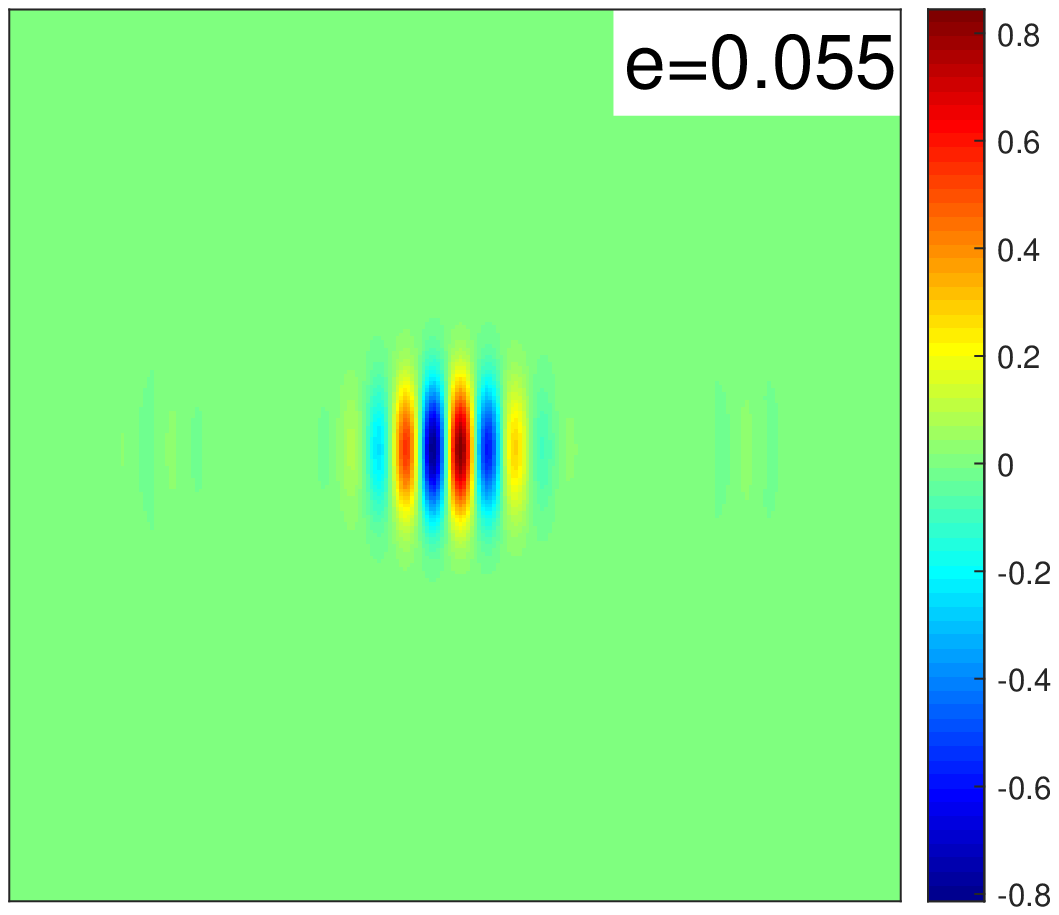}
        \subcaption{$f^{(100)}$}
    \end{subfigure}
    \caption{Reconstruction for a modified Gaussian function $f$, where $e = \frac{\|f - f^{(100)}\|_2}{\|f\|_2}$ is the relative error.}\label{fig_2}
\end{figure}
In the numerical experiment, suppose $f$ is compactly supported in the disk $D$ that is centered at the origin with radius $3$. 
When we use Landweber iteration, 
after performing the backprojection operator in each step, we smoothly cut the function such that it is still supported in radius $3$. 
More explicitly, we consider the operator
$
\mathcal{L}_\varphi = \varphi \mathcal{L} =\varphi \Lambda \chi I^*I,
$ 
where $\varphi$ is a smooth cutoff function with $\varphi = 1$ in $D$ and supported in a slightly larger disk $D'$. 
With $f$ compacted supported in $D$, we have 
$
\mathcal{L}_\varphi: L^2(D') \rightarrow L^2(D'). 
$
In this case, to solve the equation ${{I}} f = g$, we write 
\begin{equation*}
(\text{Id} - (\text{Id}- \gamma \mathcal{L}_\varphi^* \mathcal{L}_\varphi)) f = \gamma \mathcal{L}_\varphi^*(\varphi\Lambda\chi I^* g)
\end{equation*}
and use the truncated Neumann series
\[
f^{(N)} = \sum_{k=0}^{N}(\text{Id}- \gamma \mathcal{L}^*_\varphi \mathcal{L}_\varphi)^k 
\gamma \mathcal{L}^*_\varphi (\varphi\Lambda\chi I^* g).
\]
This series converges to the minimal norm solution to $\mathcal{L}_\varphi f = \varphi\Lambda\chi I^* g$ in $L^2(D')$.
After 100 steps of iteration, we get a quite good reconstruction (with the relative error $e = 0.055$).
This illustrates the result of Proposition \ref{pp_m}.

\begin{footnotesize}
    \bibliographystyle{plain}
    \bibliography{microlocal_analysis}
\end{footnotesize}

\end{document}